\documentclass[12pt]{amsart}
\usepackage[dvips]{graphicx}
\usepackage{amsmath,graphics}
\usepackage{amsfonts,amssymb,hyperref}
\usepackage{xypic}
\usepackage{comment}
\specialcomment{proofc}{}{}
\includecomment{proofc}
\theoremstyle{plain}
\newtheorem*{theorem*}{Theorem}
\newtheorem*{lemma*} {Lemma}
\newtheorem*{corollary*} {Corollary}
\newtheorem*{proposition*}{Proposition}
\newtheorem*{conjecture*}{Conjecture}
\newtheorem{theorem}{Theorem}[section]
\newtheorem{lemma}[theorem]{Lemma}
\newtheorem*{theorem1*}{Theorem 1}
\newtheorem*{theorem2*}{Theorem 2}
\newtheorem*{theorem3*}{Theorem 3}
\newtheorem{corollary}[theorem]{Corollary}
\newtheorem{proposition}[theorem]{Proposition}

\theoremstyle{remark}

\newtheorem{example*}{Example}
\newtheorem*{claim}{Claim}

\theoremstyle{definition}

\def\op{\operatorname}

\def\G{\Gamma}

\def\coker{\op{coker}}

 \def\Q{\Bbb{Q}}  \def\Z{\Bbb{Z}} \def\R{\Bbb{R}} 
\def\N{\Bbb{N}}    
 \def\a{\alpha}   \def\bp{\begin{pmatrix}}
\def\sm{\setminus} \def\ep{\end{pmatrix}} \def\bn{\begin{enumerate}} 
   \def\en{\end{enumerate}}
\def\ba{\begin{array}} \def\ea{\end{array}}  
 \def\S{\Sigma}  \def\a{\alpha} \def\b{\beta} \def\wti{\wtilde}
\def\id{\op{id}}   
  
\def\ker{\op{ker}}\def\be{\begin{equation}} \def\ee{\end{equation}} 
   
 \def\hom{\op{Hom}}  
   
 \def\dim{\op{dim}}

\def\ol{\overline}

\def\wti{\widetilde}

\def\what{\widehat}

\def\ti{\tilde}
\def\eps{\epsilon}

\def\F{\mathbb{F}}

\def\CC{\mathcal{C}}

\begin{document}
\title{The virtual Thurston seminorm of 3-manifolds}
\author{Michel Boileau}
\address{
Aix Marseille Univ, CNRS, Centrale Marseille, I2M, Marseille, France}
\email{michel.boileau@univ-amu.fr}

\author{Stefan Friedl}
\address{Fakult\"at f\"ur Mathematik\\ Universit\"at Regensburg\\   Germany}
\email{sfriedl@gmail.com}

\begin{abstract}
We show that the Thurston seminorms of all finite  covers of an aspherical 3-manifold determine whether it is  a graph manifold, a mixed 3-manifold or hyperbolic. 
\end{abstract}

\keywords{3-manifolds, Thurston norm, geometric structures on 3-manifolds}  \subjclass[2010]{57M05, 57M10, 57M27}

\maketitle

%===================================================
\section{Introduction}
Let $N$ be a 3-manifold. (Here, and throughout the paper all 3-manifolds are understood to be compact, orientable, connected, aspherical and with empty or toroidal boundary.)
Given a surface $\S$ with connected components $\S_1,\dots,\S_k$  its complexity
is defined to be
\[ \chi_-(\S):=\sum_{i=1}^d \max\{-\chi(\S_i),0\}.\]
Given a 3-manifold $N$ and $\phi \in H^1(N;\Z)$ the Thurston norm is defined as
\[ x_N(\phi):= \min\{\chi_-(\S)\, |\, \S \subset N\mbox{ is a properly embedded surface, dual to }\phi\}.\]
 Thurston
\cite{Th86} showed that  $x_N$ is a seminorm on $H^1(N;\Z)$. It follows from standard arguments that $x_N$ extends to a seminorm on $H^1(N;\R)$.
If $N$ is hyperbolic, then $N$ is in particular atoroidal which implies easily that $x_N$ is a norm. On the other hand, the seminorm is degenerate whenever there is a non-separating torus, e.g.\ if $N=S^1\times \Sigma$ where $\Sigma$ is a surface of genus $g\geq 1$.
Given any seminorm $x$ on a vector space $V$ the set $\{v\in V\,|\,x(v)=0\}$ is a subspace that we refer to as the \emph{kernel $\ker(x)$ of $x$}.

In this paper we study to which degree the Thurston norm of all finite covers of a 3-manifold determines the type of the JSJ-decomposition of the 3-manifold. Hereby we distinguish the following three mutually exclusive types of JSJ-decompositions a prime 3-manifold $N$ can have:
\bn
\item The 3-manifold $N$ is hyperbolic.
\item The 3-manifold $N$ is a graph manifold, i.e.\ all its JSJ-components are Seifert fibered spaces.
\item Following \cite{PW12} we say that $N$ is \emph{mixed} if it is if the JSJ-decomposition is non-trivial and if it contains at least one hyperbolic JSJ-component.
\en
This question is related to the general study of properties or invariants of a 3-manifold  that can be determined from its finite covers, see for example \cite{BF15},  \cite{BR15}, \cite{Le14}
 \cite{Wil16}, \cite{WZ17}.

In order to state our first result we introduce a few more definitions.
Given a $3$-manifold $N$ we write 
\begin{align*}
 \ba{rcl}\hspace{4cm} b_1(N)&:=&\dim_{\R}(H_1(N;\R)),\\
k(N)&:=& \dim_{\R}(\ker(x_N)),\\
r(N)&:=&\left\{\ba{ll} 0,&\mbox{ if }b_1(N)=0, \\
\frac{k(N)}{b_1(N)},&\mbox{ if }b_1(N)>0.\ea\right.
\intertext{Furthermore we write}
\CC(N)&:=&\mbox{the class of all finite regular covers $\ti{N}$ of $N$}, \\
\intertext{and}
\what{r}(N) &:=&\displaystyle \sup_{\wti{N}\in \CC(N)} r(\wti{N})
\ea
\end{align*}

The following proposition is well-known to the experts.

\begin{proposition}\label{prop:detectshyperbolic}
Let $N$ be an aspherical $3$-manifold with empty or toroidal boundary. Then $N$ is hyperbolic if and only if $\what{r}(N)=0$. 
\end{proposition}

\begin{proof} If $N$ is hyperbolic, then all its finite covers are hyperbolic, and as we pointed out above, in this case the seminorm is always a norm. On the other hand, if $N$ is not hyperbolic and aspherical, then by standard arguments, see e.g.\ \cite[(C.10)-(C.15)]{AFW15} there exists a finite regular cover $\wti{N}$ with a homologically essential torus. In particular $r(\wti{N})>0$. 
\end{proof}

It is harder to distinguish graph manifolds from manifolds with a non-trivial JSJ-decomposition that contain at least one hyperbolic JSJ-component. 
In order to distinguish these two classes of 3-manifolds, we need to consider a wider class of finite coverings, which we call subregular, since they correspond to subnormal subgroups of the fundamental groups. We say that a covering $f\colon \what{N}\to N$ is \emph{subregular} if the covering $f$ can be written as a composition of coverings $f_i\colon N_i\to N_{i-1}$, $i=1,\dots,k$ with $N_k=\what{N}$ and $N_0=N$, such that each $f_i$ is regular.

For a $3$-manifold $N$ we define:
\[
\ba{rcl}
\CC^{sub}(N)&:=&\mbox{the class of all finite subregular covers $\what{N}$ of $N$,}\\
\rho(N)&:=&\displaystyle\inf_{\what{N}\in \CC^{sub}(N)}r(\what{N}).\\
\what{\rho}(N)&:=&\displaystyle  \sup_{\wti{N}\in \CC(N)} \rho(\wti{N}).\ea
\] \
The following is the main result  of this paper. It characterizes graph manifolds $N$ in term of the invariant $\what{\rho}(N)$. It also gives a characterization of manifolds with non vanishing simplicial volume (i.e. with at least one hyperbolic JSJ-component). This characterization is  analogous to the one for hyperbolic manifolds in Proposition \ref{prop:detectshyperbolic}, but this time we use  the invariant $\rho(N)$ instead of $r(N)$.

\begin{theorem}\label{mainthm}
Let $N$ be an aspherical $3$-manifold with empty or toroidal boundary. 
\bn
\item If $N$ is a graph manifold, then $\what{\rho}(N)=1$. 
\item If $N$ is not a graph manifold, i.e.\  if $N$  admits a hyperbolic piece in its JSJ-decomposition, then $\what{\rho}(N)=0$. 
\en
\end{theorem}
 
The proof of Theorem~\ref{mainthm} relies on the work of Agol \cite{Ag08,Ag13}, Przytycki--Wise \cite{PW12} and Wise \cite{Wi09,Wi12a,Wi12b}.

The next corollary is a consequence of the combination of Proposition~\ref{prop:detectshyperbolic} and Theorem~\ref{mainthm}:

\begin{corollary} Let $N$ be an  aspherical $3$-manifold with empty or toroidal boundary. Then the Thurston norms of all finite subregular covers of $N$ determine into which of the following three categories $N$ falls:
\bn
\item graph manifold if and only if $\what{\rho}(N)=1$,
\item mixed manifold if and only if $\what{r}(N) >\what{\rho}(N)=0$.
\item hyperbolic manifold if and only if $\what{r}(N)=0$.
\en

\end{corollary}

\subsection*{Convention.} Unless it says specifically otherwise, all   3-manifolds are assumed to be compact, orientable, connected, and with empty or toroidal boundary.
Furthermore all surfaces are assumed to be compact and orientable. Finally,  all subsurfaces of a 3-manifold are assumed to  be properly embedded. 

\subsection*{Acknowledgment.} The authors are grateful to Aix-Marseille Universit\'e and   Universit\"at Regensburg for their hospitality. The first author was supported by ANR (projects 12-BS01-0003-01 and  12-BS01-0004-01), and the second author was supported by the SFB 1085 `Higher Invariants' at the University of Regensburg, funded by the Deutsche Forschungsgemeinschaft (DFG). Both authors  also benefited from the support and the hospitality  of  the Isaac Newton Institute for Mathematical Sciences during the programme Homology Theories in Low Dimensional Topology supported by EPSRC Grant Number EP/KO32208/1. Finally we wish to thank the referee for pointing out several minor inaccuracies.

%=================================================================
\section{The calculation of $\rho$ for graph manifolds}
The following theorem immediately implies Theorem~\ref{mainthm}~(1).

\begin{theorem}\label{mainthma}
Let $N$ be an aspherical graph manifold. Then given any $\eps>0$ there exists a finite regular cover $\what{N}$ of $N$ such that for any finite cover $\ol{N}$ of $\what{N}$ we have $r(\ol{N})>1-\eps$. 
\end{theorem}

The proof of Theorem~\ref{mainthma} will require the remainder of this section.
Given a compact manifold $X$ we write 
\[ c(X):=\op{dim}_\R\left(\coker\{H_1(\partial X;\R)\to H_1(X;\R)\}\right).\]
On several occasions we will need the following lemma.

\begin{lemma}\label{lem:cgoesup}
Let $p\colon \wti{X}\to X$ be a finite covering of a manifold $X$. Then $c(\wti{X})\geq c(X)$.
\end{lemma}

\begin{proof}
We consider the following commutative diagram of exact sequences
\[ 
\xymatrix@C-0.0cm{H_1(\partial \wti{X};\R)\ar[d]^{p_*}\ar[r]& H_1(\wti{X};\R)\ar[d]^{p_*}\ar[r]&\coker\{H_1(\partial \wti{X};\R)\to H_1(\wti{X};\R)\}\ar[d]^{p_*}\ar[r]&0\\
H_1(\partial {X};\R)\ar[r]& H_1({X};\R)\ar[r]&\coker\{H_1(\partial {X};\R)\to H_1({X};\R)\}\ar[r]&0.}
\]
For the left two vertical maps we also have the transfer maps $p^*$ going from the bottom to the top. These maps have the property that the compositions $p_*\circ p^*$ are multiplication by $[X:\wti{X}]$, in particular the transfer maps are injective.
Furthermore, the transfer maps give rise to a commutative diagram on the left. A straightforward diagram chase shows that the right vertical map also has a transfer map $p^*$ such that the composition $p_*\circ p^*$ is injective. 
\end{proof}

The next lemma is an immediate consequence of the K\"unneth Theorem.

\begin{lemma}\label{lem:cproduct}
For any surface $\Sigma$ we have $c(S^1\times \Sigma)=c(\Sigma)$.
\end{lemma}

We say that a graph manifold $N$ is of \emph{product type} if  each JSJ-component $N_v$ is a product $S^1\times \S_v$ where $\Sigma_v$ is a surface with $\chi(\S_v)<0$ and with at least two boundary components. 
%On several occasions we will make use of the fact that a finite cover of a manifold of product type is again of product type.

\begin{proposition}\label{prop:producttype}
Let  $N$ be a graph manifold that is not a Seifert fibered space and that is not finitely covered by a torus bundle.  Let $C>0$.  Then $N$ is covered by a graph manifold $\what{N}$ of product type such that for each JSJ-component $N_v$ of $\what{N}$ we have  $c(N_v)>C$.
\end{proposition}

\begin{proof}
Let  $N$ be a graph manifold that is not a Seifert fibered space and that is not finitely covered by a torus bundle. Let $C>0$.  By  \cite[Section~4.3]{AF13} (see also~\cite[(C.10)]{AFW15} and~\cite{He87}) there exists a finite cover $N'$ that is of product type.

Furthermore, by \cite[Proposition~5.22]{AF13} there exists a prime $p\geq C$ and a finite cover $N''$ of $N'$ such that for each JSJ-component $N''_v=S^1\times \S''_v$ the map $H_1(N''_v;\F_p)\to H_1(N'';\F_p)$ is injective. We denote by $\what{N}$ the finite cover of $N''$ that corresponds to the kernel of 
$\pi_1(N'')\to H_1(N'';\Z)\to H_1(N'';\F_p)$. In light of Lemma~\ref{lem:cproduct} it suffices to prove  the following claim.

\begin{claim}
Each JSJ-component of $\what{N}$ is of the form $S^1\times \Sigma$ where $\Sigma$ is a surface with $c(\Sigma)>C$.
\end{claim}

By Proposition~1.9.2 and Theorem~1.9.3 of \cite{AFW15} the JSJ-decomposition of $\what{N}$ is the pull-back of the JSJ-decomposition of $N''$. It follows from this fact and the above discussion of the chosen group homomorphism that  each JSJ-component of $\what{N}$ is the finite cover of a manifold of the form $S^1\times \Sigma$, where $\Sigma$ is a surface with at least two boundary components and with $\chi(\Sigma)<0$, and where we consider the cover corresponding to the kernel of the group homomorphism $\pi_1(S^1\times \Sigma)\to H_1(S^1\times \Sigma;\F_p)$. Note that this cover is of the form  $S^1\times \what{\Sigma}$ where  $\what{\Sigma}$ is the finite cover of $\Sigma$ corresponding to the kernel of the group homomorphism $\pi_1(\Sigma)\to H_1(\Sigma;\F_p)$.
We write $d=|H_1(\Sigma;\F_p)|$. Since $\chi(\Sigma)<0$ we have $d\geq p^2$. We make the following observations:
\bn
\item By definition of `product type' the surface $\Sigma$ has at least two boundary components. It follows that every  boundary component of $\Sigma$ has image of order precisely $p$ in $H_1(\Sigma;\F_p)$. Therefore 
\[b_0(\partial \what{\Sigma})=\frac{d}{p}\cdot b_0(\partial \Sigma).\]
\item By the multiplicativity of the Euler characteristic we have
\[ b_1(\what{\Sigma})-1=d\cdot (b_1(\Sigma)-1).\]
\item For any surface $\Sigma$ we have 
\[ b_0(\partial \Sigma)\,\,=\,\, b_1(\partial \Sigma)\,\,\leq\,\, b_1(\Sigma)+1.\]
\en
We now obtain that
\[ \ba{rcl}c(\what{\Sigma})&=&\dim_{\R}\Big( \coker\{H_1(\partial \what{\Sigma};\R)\to H_1(\what{\Sigma};\R)\}\Big)\\
&\geq & b_1(\what{\Sigma})-b_1(\partial \what{\Sigma})\\
&\geq & d(b_1(\Sigma)-1)+1-b_0(\partial \what{\Sigma})\\
&\geq & d(b_1(\Sigma)-1)+1-\frac{d}{p}(b_1(\Sigma)+1)\\
&= & d(b_1(\Sigma)-1)+1-\frac{d}{p}(b_1(\Sigma)-1)-\frac{2d}{p}\\
&\geq & -(d-\frac{d}{p})\chi(\Sigma)\\
&\geq & d-\frac{d}{p}.\ea\]
Hereby the first equality is given by definition, the following inequality is obvious, the next inequality is given by (2) and the fact that the boundary components of a surface are circles,  the following equality stems from (1) and (3), the next equality is purely algebraic, the following inequality is a consequence of $ \chi(\Sigma)=b_0(\Sigma)-b_1(\Sigma)$ and $d\geq p^2$,
 and the final inequality comes from $\chi(\Sigma)\leq -1$.

Summarizing we have shown that $c(\what{\Sigma})\geq d-\frac{d}{p}$. 
But since $d\geq p^2$ we see that the last term is at least $p\geq C$. Thus we have shown that $c(\what{\Sigma})\geq C$. 
\end{proof}

For the record we also mention the following elementary lemma.

\begin{lemma}\label{lem:regularcover}
If $f\colon \what{N}\to N$ is a finite covering of a 3-manifold, then there exists a finite regular covering $g\colon \ol{N}\to N$ that factors through $f$. 
\end{lemma}

We are now in a position to prove Theorem~\ref{mainthma}.

\begin{proof}[Proof of Theorem~\ref{mainthma}]
Let $N$ be an aspherical graph manifold and let $\eps>0$.

If $N$ is covered  by a torus bundle, then there exists a finite regular cover $\wti{N}$ with vanishing Thurston norm and with $b_1(\wti{N})\geq 1$. In particular there exists a finite regular cover $\wti{N}$ with $r(\wti{N})=1$.

If $N$ is Seifert fibered, then there exists a finite regular cover $\wti{N}$ that is an $S^1$-bundle over a surface $\Sigma$. (See~\cite[Section~4.3]{AF13} and~\cite{He87} for details.) Since $N$ is aspherical we know that $\Sigma$ is not a sphere.
The Thurston norm evidently vanishes if $\Sigma$ is a disk, or if it is an annulus, or if it is a torus, i.e.\ in these cases we have $r(\wti{N})=1$. Thus we can now suppose that $\chi(\Sigma)< 0$. 

If $\wti{N}$ is a non-trivial $S^1$-bundle  over $\Sigma$, then $\Sigma$ is closed and it follows from $\chi(\Sigma)<0$, that  $b_1(\wti{N})\geq 1$. Furthermore it is straightforward to see that all homology classes are represented by tori, thus $k(\wti{N})=b_1(\wti{N})$ and we see that $r(\wti{N})=1$.

On the other hand, if $\wti{N}$ is a trivial $S^1$-bundle over $\Sigma$, then $\wti{N}=S^1\times \Sigma$. In that case it is well-known that  $k(S^1\times \Sigma)=b_1(\Sigma)$. Since $\chi(\Sigma)<0$ there exists  a cover $S^1\times \ol{\Sigma}$ of $S^1\times \Sigma$ with $r(S^1\times \ol{\Sigma}) > 1 - \eps$. Furthermore, using Lemma~\ref{lem:regularcover} we can arrange that $S^1\times \ol{\Sigma}$ is in fact a regular cover of $N$.

For the remainder of the proof we can now assume that $N$ is neither covered by a torus bundle nor is it Seifert fibered.
It follows from  Proposition~\ref{prop:producttype}
and Lemmas~\ref{lem:cgoesup} and~\ref{lem:regularcover}
that there exists a finite regular cover $\what{N}$ of $N$
such that  $\what{N}$ is of product type and such such that for each JSJ-component $N_v$ of $\what{N}$ we have 
$c(N_v)>\frac{1}{\eps}$.
Now let $\ol{N}$ be a finite cover of $\what{N}$. As above, the JSJ-decomposition of $\ol{N}$ is induced by the JSJ-decomposition of $\what{N}$. It is thus again of product type.

We denote the JSJ-components of $\ol{N}$ by $\ol{N}_v=S^1\times \Sigma_v$, $v\in V$. It follows from Lemma~\ref{lem:cgoesup} and from the above that for each JSJ-component $\ol{N}_v$ we have $c(\ol{N}_v)>\frac{1}{\eps}$.
  For each $v$ we  denote by $f_v\in H_1(\ol{N};\Z)$ the element determined by the $S^1$-factor.

It follows from~\cite[Proposition~3.5]{EN85} and the standard calculation of the Thurston norm for products $S^1\times \Sigma$ that for any $\phi\in H^1(\ol{N};\R)$ the Thurston norm is given by
\[ x_{\ol{N}}(\phi)=\sum_{v\in V} |\phi(f_v)|\cdot \chi_-(\Sigma_v).\]
In particular, the Thurston norm vanishes if $\phi$ vanishes on all elements $f_v, v\in V$. We thus see that 
\[ k(\ol{N})\geq b_1(\ol{N})-|V|.\]
On the other hand, it follows from the Mayer--Vietoris sequence corresponding to the decomposition of $\ol{N}$ along the JSJ-tori into the JSJ-components that 
\[ b_1(\ol{N})\geq \sum_{v\in V} \dim_\R\big(\coker\{H_1(\partial \ol{N}_v;\R)\to H_1(\ol{N}_v;\R)\}\big)>\frac{1}{\eps}\cdot |V|.\]
Putting the last two inequalities together we see that 
\[ 1-r(\ol{N})\leq \frac{b_1(\ol{N})-k(\ol{N})}{b_1(\ol{N})}\leq \frac{|V|}{\frac{1}{\eps}|V|}=\eps.\]
\end{proof}

%=================================================================
\section{The calculation of $\rho$ for non-graph manifolds}

The goal of this section is to prove the following theorem, which together with  Theorem~\ref{mainthma} implies Theorem~\ref{mainthm}, since the property of being aspherical and not being a graph manifold is preserved by going to finite covers. 

\begin{theorem}\label{mainthmb}
Let $N$ be an aspherical $3$-manifold with empty or toroidal boundary that is not a graph manifold. Then given any $\eps>0$,  there exists a finite subregular
cover $\ol{N}$ of $N$ such that $r(\ol{N})<\eps$. In particular $\rho(N)=0$.
\end{theorem}

We introduce the following definitions:
\bn
\item Let $N$ be a  3-manifold. An integral class $\phi \in H^1(N;\Z)=\hom(\pi_1(N),\Z)$ is called \emph{fibered}  if  there exists a fibration $p\colon N\to S^1$ with
$\phi=p_*\colon \pi_1(N)\to \Z$.  We say \emph{$N$ is fibered} if $N$ admits a fibered class.
\item We say that a homomorphism $\phi\colon \pi \to \Z$ is \emph{large} if $\phi$ is non-trivial and if it factors through an epimorphism from $\pi$ onto a non-cyclic free group.
\en
In the following proofs we will several times make use of the followings facts:
\bn
\item[(A)] If $p\colon \wti{M}\to M$ is a finite cover and $\phi\in H^1(M;\Z)$ is a fibered class, then $p^*\phi\in H^1(\wti{M};\Z)$ is also fibered. In particular, if $M$ is fibered, then $\wti{M}$ is also fibered.
\item[(B)] If $p\colon \wti{M}\to M$ is a finite cover and $\phi\colon \pi_1(M)\to \Z$ is large, then the composition $\phi\circ p_*\colon \pi_1(\wti{M})\to \Z$ is also large.
\en
Here the first statement is obvious and the second statement follows from the fact that any finite-index subgroup of a non-cyclic free group is again a non-cyclic free group.

One key ingredient in the proof of Theorem~\ref{mainthmb} is the Virtual Fibering theorem for non-graph manifolds that is due to Agol~\cite{Ag08,Ag13}, Przytycki--Wise~\cite{PW12} and Wise~\cite{Wi09,Wi12a,Wi12b}. We refer to~\cite{AFW15} for precise references.
(See also \cite{CF17,GM17} and \cite{FK14} for alternative proofs.)

\begin{theorem} \label{thm:virtfib}\textbf{\emph{(Virtual Fibering Theorem)}}
Any  aspherical $3$-manifold that is not a graph manifold admits a finite regular cover that is fibered.
\end{theorem}

Before we continue we want to clarify our language for the JSJ-decomposition.
Let $N$ be an aspherical 3-manifold.
\bn
\item We refer to the collection of the JSJ-tori together with the boundary tori as the \emph{characteristic tori of $N$}.
\item Given an aspherical 3-manifold $N$ with boundary tori $S_1,\dots,S_k$ and JSJ-tori $T_1,\dots,T_l$ we pick disjoint tubular neighborhoods $S_i\times [-1,0]$, $i=1,\dots,k$ and $T_i\times [-1,1]$, $i=1,\dots,l$ and we refer to the components of 
\[ \textstyle N\,\,\sm \,\,\bigcup\limits_{i=1}^k S_i\times (-\frac{1}{2},0] \,\,\sm\,\, \bigcup\limits_{i=1}^l T_i\times (-\frac{1}{2},\frac{1}{2})\] as the JSJ-components of $N$. 
In particular, the complement of the union of the JSJ-components consists of tubular neighborhoods of all the characteristic tori.
\en

On two occasions we will make use of the following lemma.

\begin{lemma}\label{lem:efficient}
Let $N$ be a 3-manifold and let $N_v$ be a JSJ-component of $N$.
If $\wti{N}_v$ is a finite cover of $N_v$, then there exists a finite regular covering $p\colon N'\to N$ such that each component of $p^{-1}(N_v)$ is a finite covering of $\wti{N}_v$. 
\end{lemma}

For closed 3-manifolds this is a result of  Wilton--Zalesskii~\cite[Theorem~A]{WZ10}. The case of 3-manifolds with non-trivial boundary can easily be reduced to the closed case (see e.g.\ \cite[(C.35)]{AFW15} for details).

We continue with the following lemma.

\begin{lemma}\label{lem:a}
Let $N$ be a $3$-manifold that is not a graph manifold. Then $N$ admits a finite regular cover $N'$ such that there exists a hyperbolic JSJ-component $N'_h$ with $c(N'_h)>0$. 
\end{lemma}

\begin{proof}
Let $N_h$ be a a hyperbolic JSJ-component of $N$. It follows from the work of Agol \cite{Ag13} and Wise \cite{Wi09,Wi12a,Wi12b} (see also \cite[Flowchart~4]{AFW15} for details) that $\pi_1(N_h)$ is large, i.e.\ $\pi_1(N_h)$ admits a finite index subgroup that surjects onto a non-cyclic free subgroup. This implies, see e.g.\ \cite[(C.17)]{AFW15}, that $N_h$ admits a finite-index cover $\wti{N}_h$ with  $c(\wti{N}_h)>0$. Thus the lemma is an immediate consequence of Lemmas~\ref{lem:cgoesup} and~\ref{lem:efficient}.
\end{proof}

We also have the following lemma which might be of independent interest.

\begin{lemma}\label{lem:factorsthroughf2}
Let $N$ be a $3$-manifold and let $\phi\in H^1(N;\Z)$ be a non-trivial non-fibered class. Then there exists a finite  regular covering $p\colon N'\to N$ such that the composition $\phi\circ p_*\colon \pi_1(N')\to \Z$ is large.
\end{lemma}

The proof of the lemma is closely related to the proof of the main theorems of~\cite{FV08} and of \cite{DFV14} and to \cite[Proof~of~Theorem~3.2.4]{LR05}.

\begin{proof}
We start out with a simple observation. Let $\S$ be a surface (not necessarily connected) in a $3$-manifold dual to a class $\psi\in H^1(M;\Z)=\hom(\pi_1(M),\Z)$. We denote by $\Gamma(\Sigma)$ the graph whose vertices are precisely the components of $M$ cut along $\Sigma$ and whose edges are the components of $\Sigma$ with the obvious maps from the edges to the vertices. Then the map $\psi\colon \pi_1(M)\to \Z$ factors through the canonical epimorphism $\pi_1(M)\to \pi_1(\Gamma(\Sigma))$.

Now we turn to the proof of the lemma. 
It is clear that it suffices to prove the lemma for primitive classes.
We pick a Thurston norm minimizing surface $\S$ dual to $\phi$ that has the minimal number of components among all Thurston norm minimizing surfaces dual to $\phi$. In particular $\Sigma$ has no components that are separating. It follows easily that $\chi(\Gamma(\Sigma))\leq 0$. If $\chi(\Gamma(\S))<0$, then we are done by the above observation.

Now suppose that $\chi(\G(\S))=0$. Since $\phi$ is primitive and since $\Sigma$ has the minimal number of components it follows from the argument on \cite[p.~73]{DFV14} that $\Sigma$ is connected. By Przytycki--Wise~\cite[Theorem~1.1]{PW14} the subgroup $\pi_1(\Sigma)\subset \pi_1(M)$ is separable, i.e.\ given any $g\not\in \pi_1(\Sigma)$ there exists a homomorphism $\a\colon \pi_1(M)\to G$ onto a finite group such that $\a(g)\not\in \a(\pi_1(\Sigma))$. Since $\Sigma$ is not a fiber there exists by \cite[Theorem~10.5]{He76} a $g\in \pi_1(M\sm \Sigma\times (0,1))$ that does not come from $\pi_1(\Sigma\times \{0\})$. It now follows from a standard argument, see e.g.\ \cite[(C.15)]{AFW15} or \cite[Proof~of~Theorem~3.2.4]{LR05}, that applying subgroup separability to this $g$ allows to build an epimorphism of $\pi_1(M)$ onto a free product with amalgamation of finite groups.  The fact that the target group is virtually a free group of rank two  
gives the desired statement.
\end{proof}

%More generally, we say that a class $\phi \in H^1(N;\R)$ is  \emph{fibered}  if $\phi$ can be represented by a nowhere-vanishing closed 1--form. By \cite{Ti70}  the two notions of being fibered coincide for integral cohomology classes.

\begin{lemma}\label{lem:pq}
Let $N$ be a hyperbolic $3$-manifold and let $\alpha,\beta\in H^1(N;\Z)$ be linearly independent. Then there exist $p,q\in \Z\sm \{0\}$ such that $p\alpha+q\beta$ is not fibered.
\end{lemma}

\begin{proof}
We say that a rational class $\phi\in H^1(N;\Q)$ is fibered if some integral multiple $n\phi\in H^1(N;\Z)$, $n\in \N$ is fibered. We denote by
\[ B:=\{\phi\in H^1(N;\Q)\,|\, x_N(\phi)\leq 1\}\]
the norm ball of the Thurston seminorm. Since $x_N$ is a seminorm the set $B$ is convex and non-degenerate, the latter meaning that it is not contained in a lower-dimensional subspace of $H^1(N;\Q)$.  By assumption $N$ is hyperbolic, this implies that the Thurston seminorm on $H^1(N;\Q)$ is in fact a norm, i.e.\ $B$ is compact. Thurston \cite{Th86} showed that $B$ is a polyhedron with rational vertices. Furthermore he showed that the set of fibered classes is given by the union of cones on certain  open top-dimensional faces of the polyhedron $B$.

Now we denote by $V$ the  subspace of $H^1(N;\Q)$ spanned by $\alpha$ and $\beta$. By assumption $V$ is 2-dimensional. The intersection $B\cap V$ is a compact polytope in $V$ with rational vertices. Since the polytope $B\cap V$ is compact and non-degenerate it has at least three vertices. By the aforementioned result of Thurston any class in the cone of any of the vertices is not fibered. Since $\a$ and $\b$ are linearly independent and since there are at least three vertices, and since the vertices are rational we can find \emph{non-zero} $p,q\in \Z\sm\{0\}$ such that $p\alpha+q\beta$ lies in the cone of one of the vertices, in particular it is not fibered. 
\end{proof}

In the following we will on several occasions make use of the following lemma which is a straightforward consequence of  Proposition~1.9.2 and Theorem~1.9.3 in \cite{AFW15}.

\begin{lemma}\label{lem:hyperbolicjsj}
Let $N$ be a prime $3$-manifold and let $N_h$ be a hyperbolic JSJ-component of $N$. Then for each finite cover $p\colon N'\to N$ all the components of $p^{-1}(N_h)$ are hyperbolic JSJ-components of $N'$.
\end{lemma}

\begin{lemma}\label{lem:mostlyfibered}
Let $N$ be a mixed 3-manifold. Then there exists a finite regular cover $N'$ of $N$, a hyperbolic JSJ-component $N_h'$ and a class $\phi\in H^1(N';\Z)$ such the restriction of $\phi$ to $N_h'$ is non-fibered but such that the restriction of $\phi$ to $\ol{N\sm N_h'}$ is fibered. 
\end{lemma}

\begin{proof}
By Theorem~\ref{thm:virtfib}, Lemmas~\ref{lem:regularcover} and~\ref{lem:a} and Observation (A) there exists a finite regular cover $N'$ of $N$ that admits a  fibered class $\phi\in H^1(N';\Z)$ and that admits  a hyperbolic JSJ-component $N'_h$ with the property that  $c(N_h')>0$. This implies that there exists a non-trivial homomorphism  $\psi_h\colon H_1(N_h';\Z)\to \Z$ that is trivial on the image of any boundary component of $N_h'$.  In particular $\psi_h$ factors through $H_1(N',\partial N';\Z)$. 
We denote the resulting homomorphism
\[ H_1(N';\Z)\to H_1(N',\ol{N'\sm N_h'};\Z)\cong H_1(N_h',\partial N_h';\Z)\xrightarrow{\psi_h} \Z\]
by $\psi$. 

We denote by $\phi_h$ the restriction of $\phi$ to $N'_h$. 
The classes $\phi_h$ and $\psi_h$ in $H^1(N'_h;\Z)$ are linearly independent since the former, as a fibered class is non-trivial on each boundary component of $N'_h$ whereas the latter is by construction trivial on each boundary component. 
By Lemma~\ref{lem:pq} there exist  $p,q\in \Z\sm \{0\}$ such that $p\phi_h+q\psi_h$ is a non-fibered class in $H^1(N_h';\Z)$. 
On the other hand, the restriction of $p\phi_h+q\psi_h$ to $\ol{N\sm N_h'}$ equals the restriction of $p\phi_h$ to $\ol{N\sm N_h'}$. Since $p\ne 0$ this is a fibered class. 
\end{proof}

\begin{lemma}\label{lem:factorsthroughf2b}
Let $N$ be a mixed 3-manifold. Then there exists a finite subregular cover $N'$ of $N$, hyperbolic JSJ-components $N_1',\dots,N_k'$, $k\geq 1$ of $N'$, and a homomorphism $\phi\in \hom(H_1(N';\Z),\Z)=H^1(N';\Z)$ such that the restriction of $\phi$ to each $N_i'$ is large  but such that the restriction of $\phi$ to $\ol{N'\sm (N_1'\cup \dots \cup N_k')}$ is fibered. 
\end{lemma}

\begin{proof}
In light of Lemma~\ref{lem:mostlyfibered} we can without loss of generality assume that there exists a hyperbolic JSJ-component $N_h$ of $N$ and a class $\phi\in H^1(N';\Z)$ such the restriction of $\phi$ to $N_h$ is non-fibered but such that the restriction of $\phi$ to $\ol{N\sm N_h}$ is fibered. 

By Lemmas~\ref{lem:factorsthroughf2} and~\ref{lem:efficient} and Observation~(B) there exists a finite regular cover $p\colon N'\to N$ such that for one (and hence all) components $N_1',\dots,N_k'$ of $p^{-1}(N_h)$ the map $p\circ \phi\colon \pi_1(N_i')\to \pi_1(N_h)\to \Z$ factors through an epimorphism onto a non-cyclic free group. 

On the other hand it follows from Observation~(A)  that the restriction of $p^*\phi$ to $\ol{N'\sm (N_1'\cup \dots \cup N_k')}=p^{-1}(\ol{N\sm N_h})$ is  fibered.
\end{proof}

Let $N$ be a 3-manifold. We have the following notations:
\bn
\item Given  $\phi\in H^1(N;\Z)=\hom(\pi_1(N),\Z)$ and $n\in \N$ we  denote by $\phi_n\colon \pi_1(N)\to \Z_n$ the  homomorphism that is given by the composition of $\phi$ with the projection map $\Z\to \Z_n$. 
\item Given a homomorphism $\a\colon \pi_1(N)\to G$ we denote by $N_\a$ the corresponding cover. If $\a$ is not surjective, then $N_\a$ consists of $|\coker(\a)|$ copies of the finite cover of $N$ corresponding to $\ker(\a)$. 
\en

We recall the following well-known lemma.

\begin{lemma}\label{lem:b1bounded}
Let $N$ be a 3-manifold and let $\phi\in H^1(N;\Z)=\hom(\pi_1(N),\Z)$ be a fibered class. Then for all but finitely many primes $p$ we have
\[ b_1(N_{\phi_p})\leq 3+x_N(\phi).\]
\end{lemma}

\begin{proof}
Let $\phi$ be a fibered class. We write $\phi=d\psi$ where $\psi$ is a primitive class and $d\in \N$.
It is well-known that $\psi$ is again fibered with $x_N(\phi)=dx_N(\psi)$. We denote by $S$ the fiber of the surface bundle corresponding to $\psi$. Surthermore we denote by $\varphi\colon \pi_1(S)\to \pi_1(S)$ the corresponding monodromy. Also, given any automorphism $\gamma$ of $\pi_1(S)$ we denote by $\Z\ltimes_{\gamma}\pi_1(S)$ the corresponding semidirect product. 

Now let $n\in \N$. It is straightforward to see that 
\[\ba{rcl} H_1(N_{\psi_n};\Z)&\cong&H_1(n\Z\ltimes_{\varphi}\pi_1(S);\Z)\\
&\cong& H_1(\Z\ltimes_{\varphi^n}\pi_1(S);\Z)\cong \Z\oplus H_1(S;\Z)/(\varphi^n-\id).\ea\]
It follows that 
\[  b_1(N_{\psi_n})\leq 
\op{rank}_{\Z}(\Z\oplus H_1(S;\Z)/(\varphi^n-\id))\leq 1+b_1(S)\leq 3+x_N(\psi).\]

Now let $p$ be a prime that is coprime to $d$.
It follows that the map
\[ \pi_1(N)\xrightarrow{d\cdot \psi=\phi}d\Z\to \Z_p\]
is surjective. In particular $N_{\phi_p}=N_{\psi_p}$, and we thus see from the above that 
\[ b_1(N_{\phi_p})=b_1(N_{\psi_p})\leq 3+x_N(\psi)\leq 3+x_N(\phi).\]
\end{proof}

The following is the last lemma that we will need for the proof of Theorem~\ref{mainthmb}.

\begin{lemma}\label{lem:f2largec}
Let $N$ be a 3-manifold and let $\phi\colon \pi_1(N)\to \Z$ be a large  homomorphism such that the restriction of $\phi$ to all boundary-components of $N$ is non-trivial.
 Then for all but finitely many primes $p$ we have
\[ c(N_{\phi_p})\geq p-1-2b_0(\partial N).\]
\end{lemma}

\begin{proof}
Let $N$ be a 3-manifold and let $\phi\colon \pi_1(N)\to \Z$ be a non-trivial homomorphism that factors through an epimorphism $\a\colon \pi_1(N)\to F$ onto a non-cyclic free group $F$ and such that the restriction of $\phi$ to all boundary-components of $N$ is non-trivial. By a slight abuse of notation we denote the induced homomorphism $F\to \Z$ by $\phi$ as well.

We denote the boundary components of $N$ by $T_1,\dots,T_k$. For each $i\in \{1,\dots,k\}$ we define $d_i\in \N$ by the condition that $\phi(\pi_1(T_i))=d_i\Z$. Similarly we define $d$ by $\phi(\pi_1(N))=d\Z$. By our hypothesis we know that $d$ and  all the  $d_i$ are non-zero.

Now let $p$ be any prime that is coprime to $d$ and to $d_1,\dots,d_k$. This choice of $p$ implies that the restriction of $\phi_p$ to each boundary component is surjective. Furthermore the homomorphism $\phi_p\colon F\to \Z_p$ is surjective. We deduce that 
\[ b_1(N_{\phi_p}) \,\,\geq \,\,\op{rank}(\ker(\phi_p\colon F\to \Z_p))\,\,\geq \,\,p-1.\]
Since the restriction of $\phi_p$ to each boundary component is surjective we see that the induced covering of each boundary component is connected. Put differently, $N_{\phi_p}$ has precisely $k$ boundary components, each of which is a torus. We conclude that 
\[ \ba{rcl}c(N_{\phi_p})&=&\op{rank}\left(\coker\{H_1(\partial N_{\phi_p};\Z)\to H_1(N_{\phi_p};\Z)\}\right)\\[0.05cm]
&
\geq &b_1(N_{\phi_p})-b_1(\partial N_{\phi_p})\,\,\geq\,\, p-1-2b_0(\partial N).\ea\]
\end{proof}

We are now in a position to prove Theorem~\ref{mainthmb}.

\begin{proof}[Proof of Theorem~\ref{mainthmb}]
Let $N$ be an aspherical $3$-manifold that is not a graph manifold. We need to show that  given any $\eps>0$,  there exists a finite subregular
 cover $\ol{N}$ of $N$ such that $r(\ol{N})<\eps$. 

So let $N$ be an aspherical 3-manifold that is not a graph manifold and let $\eps>0$.
If $N$ is hyperbolic then it follows from  Proposition~\ref{prop:detectshyperbolic} that already $r(N)=0$.
Thus henceforth we can restrict ourselves to the case that $N$ is not hyperbolic, i.e.\ $N$ is a mixed manifold. 

By Lemma~\ref{lem:factorsthroughf2b} we can without loss of generality assume that there exists $k\geq 1$ hyperbolic JSJ-components $N_1,\dots,N_k$ of $N$  and a homomorphism $\phi\in H^1(N;\Z) = \hom(H_1(N;\Z),\Z)$ such the restriction of $\phi$ to each $N_i$, $i=1,\dots,k$ is large  but such that the restriction of $\phi$ to $M:=\ol{N\sm (N_1\cup \dots \cup N_k)}$ is fibered. 

By our definition of JSJ-components we see that $M$ contains all characteristic tori of $N$. Since $\phi|_M$ is fibered it follows from \cite[Section~4]{EN85} that the restriction of $\phi$ to a tubular neighborhood of each characteristic torus is a fibered class. It follows in particular that the restriction of $\phi$ to each characteristic torus is non-zero. This in turn implies that for almost all primes $p$ the restriction of $\phi_p$ to each characteristic torus is an epimorphism. 

We write $C:=3+x_M(\phi|_M)$. We denote by $j$ the number of JSJ-tori of $N$ and we denote by $b$ the number of boundary tori of $N$. 
% and $D:=\max\{b_0(\partial N_1),\dots,b_0(\partial N_k)\}$. 
By the above and by Lemmas~\ref{lem:b1bounded}  and~\ref{lem:f2largec} there exists a prime $p$ such that the covering map 
$f\colon \ol{N}\to N$ corresponding to the homomorphism $\phi_p\colon \pi_1(N)\to \Z_p$ has the following properties:
\bn
\item The restriction of $\phi_p$ to each characteristic torus and to each JSJ-component is an epimorphism. In particular the preimages of the JSJ-tori and the JSJ-components under $f$ are connected.
\item For each $i\in \{1,\dots,k\}$ we have  $c(f^{-1}(N_i))>\frac{C+6j+2b}{k\eps}$.
\item We have $b_1(f^{-1}(M))\leq C$. 
\en 
We claim that $\ol{N}$ has the desired property. 

It follows from the Mayer--Vietoris sequence applied to the decomposition of $\ol{N}$ along the $j$ tori that are given by the preimages of the JSJ-tori of $N$ and from (3) that 
\[\sum_{i=1}^k c(f^{-1}(N_i))\,\,\leq\,\,  b_1(\ol{N})\leq C+2j+\sum_{i=1}^k b_1(f^{-1}(N_i)).\]
The union of the  $f^{-1}(N_i)$, $i=1,\dots,k$ has at most  $2j+b$ boundary tori.
It follows easily that 
\[ \sum_{i=1}^k b_1(f^{-1}(N_i))\,\,\leq \,\,4j+2b+\sum_{i=1}^k c(f^{-1}(N_i)).\]
Putting the above two inequalities together we obtain that 
\[\sum_{i=1}^k c(f^{-1}(N_i))\,\,\leq \,\, b_1(\ol{N})\,\,\leq \,\,C+6j+2b+\sum_{i=1}^k c(f^{-1}(N_i)).\]
On the other hand, it follows from the same Mayer--Vietoris sequence together with  the fact that the Thurston seminorm is in fact a norm on hyperbolic 3-manifolds that 
\[ k(\ol{N})\leq b_1(\ol{N})-\sum_{i=1}^kc(f^{-1}(N_i)).\]
The combination of the last two inequalities together with (2) shows that 
\[ r(\ol{N})\,\,=\,\,\frac{k(\ol{N})}{b_1(\ol{N})}\,\,\leq\,\, \frac{C+6j+2b}{\sum_{i=1}^k c(f^{-1}(N_i))}<\eps.\]

%It follows from the Mayer--Vietoris sequence applied to the decomposition of $\ol{N}$ along the $j$ tori that are given by the preimages of the JSJ-tori of $N$ and from (2) and (3) that 
%\[ b_1(\ol{N})\leq C+j+\sum_{i=1}^k b_1(f^{-1}(N_i)).\]
%The sum of the boundary tori of $b_1(f^{-1}(N_i))$, $i=1,\dots,k$ equals $2j+b$. 
%It follows easily that 
%\[ \sum_{i=1}^k b_1(f^{-1}(N_i))\leq 4j+2b+\sum_{i=1}^k c(f^{-1}(N_i)).\]
%Putting these two inequalities together we obtain that 
%\[ b_1(\ol{N})\leq 4j+2b+\sum_{i=1}^k c(f^{-1}(N_i)).\]
%On the other hand, it follows from the same Mayer--Vietoris sequence together with (2) and the fact that the Thurston seminorm is in fact a norm on hyperbolic 3-manifolds that 
%\[ k(\ol{N})\leq b_1(\ol{N})-\sum_{i=1}^kc(f^{-1}(N_i)).\]
%The combination of the last two inequalities shows that 
%\[ r(\ol{N})=\frac{k(\ol{N})}{b_1(\ol{N})}\]
\end{proof}

%=================================================================

\end{document}